\documentclass[10pt,a4paper]{amsart}
\usepackage[utf8]{inputenc}
\usepackage[T1]{fontenc}
\usepackage{amsmath}
\usepackage{upquote}
\usepackage{bm}
\usepackage{amsfonts}
\usepackage{amssymb}
\usepackage[english]{babel}
\usepackage{tabularx}
\usepackage{mathtools}
\usepackage{adjustbox}
\usepackage{graphics}
\usepackage{hyperref}
\usepackage{xcolor}
\usepackage{diagbox}
\usepackage{color}
\usepackage{tikz}
\usepackage{chngcntr}
\usepackage{array}
\newcolumntype{H}{>{\setbox0=\hbox\bgroup}c<{\egroup}@{}}
\usetikzlibrary{calc}
\numberwithin{equation}{section}
\newtheorem{thm}{Theorem}[section]
\newtheorem{pr}[thm]{Proposition}
\newtheorem{lm}[thm]{Lemma}
\newtheorem{re}[thm]{Remark}

\newtheorem{con}[thm]{Conjecture}
\newtheoremstyle{case}{}{}{}{}{}{:}{ }{}
\theoremstyle{case}
\newtheorem{case}{Case}
\counterwithin*{case}{thm}
\newtheoremstyle{caso}{}{}{}{}{}{:}{ }{}
\theoremstyle{caso}

\DeclareRobustCommand{\svdots}{% s for `scaling'
  \vbox{%
    \baselineskip=0.33333\normalbaselineskip
    \lineskiplimit=0pt
    \hbox{.}\hbox{.}\hbox{.}%
    \kern-0.2\baselineskip
  }%
}

\newcommand{\floor}[1]{\left\lfloor #1 \right\rfloor}
\newcommand{\flce}[1]{\left\lfloor #1 \right\rceil}

\theoremstyle{remark}

\newenvironment{proof*}[1]
  {\renewcommand{\proofname}{#1}%
   \begin{proof}}
  {\end{proof}}

\makeatletter
\let\@@pmod\pmod
\DeclareRobustCommand{\pmod}{\@ifstar\@pmods\@@pmod}
\def\@pmods#1{\mkern4mu({\operator@font mod}\mkern 6mu#1)}
\makeatother

\title[On the Bessenrodt-Ono type inequality for the $A$-partition function]{On the Bessenrodt-Ono type inequality for a wide class of $A$-partition functions}
\author{Krystian Gajdzica}
\address{Institute of Mathematics \\
	Faculty of Mathematics and Computer Science \\
	Jagiellonian University in Cracow
}
\email{krystian.gajdzica@doctoral.uj.edu.pl}



\begin{document}

\setlength{\parindent}{10mm}
\maketitle

\begin{abstract}
The $A$-partition function $p_A(n)$ enumerates those partitions of $n$ whose parts belong to a fixed (finite or infinite) set $A$ of positive integers. On the other hand, the extended $A$-partition function $p_A\left(\bm{\mu}\right)$ is defined as an multiplicative extension of the $A$-partition function to a function on $A$-partitions. In this paper, we investigate the Bessenrodt-Ono type inequality for a wide class of $A$-partition functions. In particular, we examine the property for both the $m$-ary partition function $b_m(n)$ and the $d$-th power partition function $p_d(n)$. Moreover, we show that $b_m(\bm{\mu})$ ($p_d(\bm{\mu})$) takes its maximum value at an explicitly described set of $m$-ary partitions (power partitions), where $\bm{\mu}$ is an $m$-ary partition (a power partition) of $n$. Additionally, we exhibit analogous results for the Fibonacci partition function and the `factorial' partition function. It is worth pointing out that an elementary combinatorial reasoning plays a crucial role in our investigation. %deal with these issues.
\end{abstract}

\section{Introduction}

A partition $\lambda=(\lambda_1,\lambda_2,\ldots,\lambda_j)$ of a non-negative integer $n$ is a weakly decreasing sequence of positive integers such that
\begin{align*}
    n=\lambda_1+\lambda_2+\cdots+\lambda_j.
\end{align*}
Elements $\lambda_i$ are called parts of the partition $\lambda$. By $P(n)$, we mean the set of all partitions of $n$. Further, the number of all partitions of $n$ is denoted by $p(n)$, i.e. $p(n)=\#P(n)$. For instance, we have that $p(4)=5$ with
\begin{align*}
    (4),\hspace{0.2cm}(3,1),\hspace{0.2cm}(2,2),\hspace{0.2cm}(2,1,1)\hspace{0.2cm}\text{ and }\hspace{0.2cm}(1,1,1,1).
\end{align*}
It is well known that the generating function for $p(n)$ takes the form
\begin{align*}
    \sum_{n=0}^\infty p(n)q^n=\prod_{i=1}^\infty\frac{1}{1-q^i}.
\end{align*}
That is Euler's result from 1748. Since then, the partition function has been investigated from various perspectives. Therefore, there is a plethora of literature devoted to its analytic, arithmetic and combinatorial properties. We refer the reader to \cite{AK, GA2, GA1, H, Sills} for a comprehensive introduction to the theory of partitions.

Now, let us recall that a sequence $\left(c_n\right)_{n=0}^\infty$ of real numbers is said to be log-concave if it fulfills $c_n^2>c_{n-1}c_{n+1}$ for every $n\geqslant1$. In 2015, DeSalvo and Pak \cite{DSP} reproved Nicolas' theorem \cite{N} concerning the log-concavity of the partition function.

\begin{thm}[DeSalvo-Pak, Nicolas]\label{Theorem: DeSalvo-Pak}\ \\
The partition function is log-concave for every $n\geqslant26$.
\end{thm}

Their paper has initiated wide research related to investigating similar phenomena for other variations of the partition function. We encourage the reader to see \cite{BKRT, Craig, Dawsey, KG2, HN2, HN7, HNT2, Liu, Ono, MU} for more results related to that subject.

However, DeSalvo and Pak's paper has also motivated to study other inequalities for partition statistics. In 2016, Bessenrodt and Ono \cite{BO} discovered the following result.

\begin{thm}[Bessenrodt-Ono]\label{Theorem: Bessenrodt-Ono (1)}\ \\
    Let $a,b\geqslant2$ be positive integers such that $a+b>9$. Then
    \begin{equation*}
        p(a)p(b)>p(a+b).
    \end{equation*}
\end{thm}

The original proof of the above identity is based on asymptotic estimates due to Lehmer \cite{L}. There are also two alternative proofs of Theorem \ref{Theorem: Bessenrodt-Ono (1)}. The first of them is in combinatorial manner due to Alanazi, Gagola and Munagi \cite{AGM}. The second one, on the other hand, is done via induction on $a+b$ by Heim and Neuhauser \cite{HN3}.

It is worth pointing out that the Bessenrodt-Ono inequality is not only art of art's sake, but it allows us to determine the maximum value of the form

\begin{equation*}
    \max{p(n)}:=\max{\left\{p(\bm{\lambda}):\bm{\lambda}\text{ is a partition of } n\right\}},
\end{equation*}
where $p(\bm{\lambda})$ denotes the extended partition function defined as
\begin{align*}
    p(\bm{\mu}):=\prod_{i=1}^jp(\mu_i)\hspace{0.2cm}\text{ for an arbitrary partition }\hspace{0.2cm}\bm{\mu}=(\mu_1,\mu_2,\ldots,\mu_j).
\end{align*}
More precisely, the following is true.

\begin{thm}[Bessenrodt-Ono]\label{Theorem: Bessenrodt-Ono (2)}\ \\
    Let $n\in\mathbb{N}_+$. For $n\geqslant4$ and $n\not=7$, the maximal value $\max{p(n)}$ of the partition function on $P(n)$ is attained at the partition
    \begin{align*}
        &\left(4,4,4,\ldots,4,4\right)\text{ when } n\equiv0\pmod*{4},\\
        &\left(5,4,4,\ldots,4,4\right)\text{ when } n\equiv1\pmod*{4},\\
        &\left(6,4,4,\ldots,4,4\right)\text{ when } n\equiv2\pmod*{4},\\
        &\left(6,5,4,\ldots,4,4\right)\text{ when } n\equiv3\pmod*{4}.
    \end{align*}
    In particular, if $n\geqslant8$, then
    \begin{align*}
        \max{p(n)}=\begin{cases}
            5^{\frac{n}{4}}, &\text{ if } n\equiv0\pmod*{4},\\
            7\cdot5^{\frac{n-5}{4}}, &\text{ if } n\equiv1\pmod*{4},\\
            11\cdot5^{\frac{n-6}{4}}, &\text{ if } n\equiv2\pmod*{4},\\
            11\cdot7\cdot5^{\frac{n-11}{4}}, &\text{ if } n\equiv0\pmod*{4}.
        \end{cases}
    \end{align*}
\end{thm}

Similarly to DeSalvo and Pak's result \cite{DSP}, Theorem \ref{Theorem: Bessenrodt-Ono (1)} and Theorem \ref{Theorem: Bessenrodt-Ono (2)} have emerged extensive research devoted to exploring analogous properties for other partition statistics. For some works related to the topic see, for instance, \cite{BB, Chern, Dawsey, KG2, KG3, HNT2, Hou, Li, Liu, Males}. 

In this paper, on the other hand, we investigate the Bessenrodt-Ono type inequality for the so-called $A$-partition function. Let us recall that for a given set (or multiset) $A$ of positive integers, the $A$-partition function $p_A(n)$ enumerates those partitions of $n$ whose parts are elements from $A$. It is well-known that the generating function for $p_A(n)$ is given by
\begin{align*}
    \sum_{n=0}^\infty p_A(n)q^n=\prod_{a\in A}\frac{1}{1-q^a}.
\end{align*}

For a various examples of $A$-partition functions, we encourage the reader to see classical Andrews' books \cite{GA2, GA1}.
%There is a wealth of literature devoted to properties of $A$-partition functions for various choices of the set $A$. We refer the reader to, for example, \cite{GA, A1, A2, Bateman, CN, KG1, Gawron, GMU, O’Sullivan, RS, Tenenbaum}.

%$m$-ary partition function. For some additional information regarding the asymptotic behavior of $b_m(n)$, we refer the reader to de Bruijn \cite{de Bruijn}, Fr\"oberg \cite{Froberg}, Mahler \cite{Mahler} and Pennington \cite{Pennington}. On the other hand, Alkauskas \cite{Al}, Andrews \cite{A1}, Andrews, Fraenkel and Sellers \cite{A2}, Churchhouse \cite{CH1, CH2}, Dirdal \cite{Dirdal}, Dolph, Reynolds and Sellers \cite{Dolph}, Edgar \cite{E}, Gawron, Miska and Ulas \cite{GMU}, Gupta \cite{G1, G2}, R{\o}dseth \cite{RS1}, R{\o}dseth and Sellers \cite{RS2}, Sobolewski and Ulas \cite{SU}, Ulas and Żmija \cite{MUŻ} and Żmija \cite{Z} explore some arithmetic properties of the function $b_m(n)$ from both combinatorial and number theoretic perspectives.

Here, we focus on some particular choices for set $A$. For instance, we investigate the Bessenrodt-Ono type inequality for the $m$-ary partition function $b_m(n)$, i.e.  the $A$-partition function $p_A(n)$ with $A=\{m^n:n\in\mathbb{N}\}$ and $m\geqslant2$. For some additional information related to $m$-ary partitions, we refer to \cite{Al, A1, Mahler, RS1, RS2, Z}. Further, we deal with the ($d$-th) power partition function $p_d(n):=p_{A_d}(n)$, where $A_d=\{n^d:n\in\mathbb{N}_+\}$ and $d\geqslant2$ (for more details, see \cite{O’Sullivan, Tenenbaum, MU}). Moreover, we also apply the obtained results to derive the analogues of Theorem \ref{Theorem: Bessenrodt-Ono (2)}. 

However, the most important result is a property which allows us to examine the validity of the Bessenrodt-Ono inequality for a rich class of $A$-partition functions.

%reduce the examination of the Bessenrodt-Ono type inequality to a finite task for a wide class of $A$-partition functions. Before we state the theorem, let us recall that $p_A(n| \text{ condition }X)$ denotes the number of $A$-partitions of $n$ that satisfy the condition $X$. 

\begin{thm}\label{Theorem: General B-O}
     Let $A=\{1,a_2,a_3,\ldots\}$ be an arbitrary (finite or infinite) set of positive integers with $1<a_2<a_3<\cdots$ such that $a_k-a_l\geqslant a_3$ for every $k>l\geqslant3$ and $2a_2\leqslant a_3$. Suppose further that the inequality
     \begin{align*}
         p_A(w)p_A(z)>p_A(w+z)
     \end{align*}
     is satisfied for all numbers $w,z\geqslant a_3$ with $2a_3\leqslant w+z\leqslant2(a_3+a_2)-1$. Then, the inequality  
    \begin{align*}
         p_A(w)p_A(z)>p_A(w+z)
     \end{align*}
    holds for every $w,z\geqslant a_3$.
\end{thm}
The usefulness of the above theorem is highlighted in the sequel. In particular, we apply it to deal with the Bessenrodt-Ono inequality for the power partition function. Finally, it is worth noting that a combinatorial reasoning plays an important role in the proof of Theorem \ref{Theorem: General B-O}.

This manuscript is organized as follows. In Section 2, we introduce necessary notation and conventions. The proof of Theorem \ref{Theorem: General B-O} is presented in Section 3. Section 4 and Section 5 deal with the Bessenrodt-Ono type inequalities for the $m$-ary partition function and the power partition function, respectively. Finally, Section 6 contains some additional applications of Theorem \ref{Theorem: General B-O} and concluding remarks.

\section{Preliminaries}

At first, we fix some notation. By  $\mathbb{N}$, $\mathbb{N}_{+}$ and $\mathbb{N}_{\geqslant k}$, we mean the set of non-negative integers, the set of positive integers and the set of positive integers greater or equal than $k$, respectively. We do not repeat the definition of the $A$-partition function, because it might be found in Introduction. Throughout the paper, we also use the following conventions. We put $P_A(n):=\{\lambda:\lambda\text{ is an }A\text{-partition of } n\}$ and define the extended $A$-partition function as
\begin{align*}
    p_A(\bm{\lambda}):=\prod_{i=1}^jp_A(\lambda_i)\hspace{0.2cm}\text{ for an arbitrary $A$-partition }\hspace{0.2cm}\bm{\lambda}=(\lambda_1,\lambda_2,\ldots,\lambda_j).
\end{align*}
We also set
\begin{align*}
    \max{p_A(n)}:=\max{\left\{p_A(\bm{\lambda}):\bm{\lambda}\text{ is an }A\text{-partition of } n\right\}}.
\end{align*}
 
Moreover, $p_A(n| \text{ condition }X)$ and $P_A(n| \text{ condition }Y)$ denote the number of $A$-partitions of $n$ that satisfy the condition $X$ and the set of $A$-partitions of $n$ that fulfill the condition $Y$, respectively. Further, the Cartesian product of two sets of partitions, let say, $U$ and $V$ is written as $U\oplus V$, i.e.
\begin{align*}
    U\oplus V=\{\left(\lambda;\mu\right):\lambda\in U, \mu\in V\}.
\end{align*}
Finally, it should be pointed out that exponents appearing in a partition refer to multiplicities, e.g. $(5^3,2^2,1)=(5,5,5,2,2,1)$.

Now, we are ready to proceed to the main part of the paper.

\section{The proof of Theorem \ref{Theorem: General B-O}}

Before we delve into the proof of Theorem \ref{Theorem: General B-O}, we need to show an auxiliary result which allows us to reduce the Bessenrodt-Ono inequality problem for $p_A(n)$ to a finite task.

\begin{thm}\label{Theorem: Injection}
     Let $A=\{1,a_2,a_3,\ldots\}$ be an arbitrary (finite or infinite) set of positive integers with $1<a_2<a_3<\cdots$ such that $a_k-a_l\geqslant a_3$ for every $k>l\geqslant3$. Then, the inequality  
    \begin{align*}
        p_A(w|\text{ no }a_2\text{'s})p_A(z)\geqslant p_A(w+z|\text{ no } a_2\text{'s})
    \end{align*}
    holds for all positive integers $w\geqslant a_3+1$ and $z\geqslant 2a_2$.
\end{thm}
\begin{proof}
    Let $a_2,a_3,w,z$ be as in the statement. Similarly to Alanazi, Gagola and Munagi \cite{AGM}, we put 
    \begin{align*}
        i=i(\lambda):=\max{\{j\in\mathbb{N}_+: \lambda_j+\lambda_{j+1}+\cdots+\lambda_{t+s}\geqslant z\}},
    \end{align*}
    where $\lambda_1\geqslant\cdots\geqslant\lambda_{t}\geqslant a_3$ and $\lambda=\left(\lambda_1,\lambda_2,\ldots,\lambda_{t+s}\right)=\left(\lambda_1,\lambda_2,\ldots,\lambda_{t},1^s\right)\in P_A(w+z|\text{ no }a_2\text{'s})$ is arbitrary. We also assume that $\lambda_i=x+y$, where $x+\lambda_{i+1}+\cdots+\lambda_{t+s}=z$ and $\lambda_1+\cdots+\lambda_{i-1}+y=w$. Hence, it follows from the definition of $i$ that $x\geqslant1$ and $y\geqslant0$.
    
    Now, our main aim is to construct an injective map 
    \begin{align*}
        f:P_A(w+z|\text{ no }a_2\text{'s})\to P_A(w|\text{ no }a_2\text{'s})\oplus P_A(z).
    \end{align*}
    In order to do that, let us assume that $\lambda=\left(\lambda_1,\lambda_2,\ldots,\lambda_t,1^s\right)\in P_A(w+z|\text{ no }a_2\text{'s})$. We construct the map in a few steps depending on $x,y$ and $s$. For the sake of convenience, we label each of the cases by its general assumptions.

     \begin{case}[$y=0 \text{ and } 0\leqslant s\leqslant z$]\label{Case 1}
    In this case, we simply put
    \begin{align*}
        f(\lambda)=
            (\lambda_1,\ldots,\lambda_{i-1};\lambda_i,\ldots,\lambda_t,1^s).
    \end{align*}
    It is clear that if $f(\lambda)=f(\mu)$ in this situation, then $\lambda=\mu$.
\end{case}

\begin{case}[$y=0 \text{ and } z< s\leqslant a_3+z$]\label{Case 2}
    Under these assumptions, we set
    \begin{align*}
        f(\lambda)=(\lambda_1,\ldots,\lambda_{t-1},1^{s-z+\lambda_{t}};1^z).
    \end{align*}
    At first, let us notice that $\lambda_{t}$ exists as a consequence of the fact that $s-z<a_3+1\leqslant w$. Thus, it follows that $\lambda_{t}\geqslant a_3$. Now, we suppose that for the partitions $\lambda=\left(\lambda_1,\lambda_2,\ldots,\lambda_{t+s}\right)=\left(\lambda_1,\lambda_2,\ldots,\lambda_{t},1^s\right)$ and $\mu=\left(\mu_1,\mu_2,\ldots,\mu_{\Tilde{t}+\Tilde{s}}\right)=\left(\mu_1,\mu_2,\ldots,\mu_{\Tilde{t}},1^{\Tilde{s}}\right)$ (which fulfill the conditions $y,\Tilde{y}=0$ and $z< s,\Tilde{s}\leqslant a_3+z$, where a tilde indicates corresponding parameters for the partition $\mu$) the following equality holds:
    \begin{align*}
        f(\lambda)=(\lambda_1,\ldots,\lambda_{t-1},1^{s-z+\lambda_{t}};1^z)=(\mu_1,\ldots,\mu_{\Tilde{t}-1},1^{\Tilde{s}-z+\mu_{\Tilde{t}}};1^z)=f(\mu).
    \end{align*}
    If that is the case, then $t=\Tilde{t}$ and $\lambda_m=\mu_{m}$ for every $1\leqslant m\leqslant t-1$. Moreover, we have that
    \begin{align*}
        s-z+\lambda_t=\Tilde{s}-z+\mu_t.
    \end{align*}
    Without loss of generality we can assume that $a_k=\lambda_t\geqslant\mu_t=a_l\geqslant a_3$. Thus, we get 
    \begin{align*}
        a_k-a_l=\lambda_t-\mu_t=(\Tilde{s}-z)-(s-z)<a_3.
    \end{align*}
    Now, it follows from the assumptions in the statement of Theorem \ref{Theorem: Injection} that $\lambda_t=\mu_t$. Hence, $s=\Tilde{s}$ and $\lambda=\mu$, as required.

    It should be also pointed out that Case \ref{Case 2} is separate from Case \ref{Case 1}. Here, notice that $1$ appears more than $z$ times as a part in $f(\lambda)$ what is impossible in Case \ref{Case 1}.
\end{case}

\begin{case}[$y=0 \text{ and } s> a_3+z$]\label{Case 3}
    In this situation, we define $f(\lambda)$ as
    \begin{align*}
        f(\lambda)=(\lambda_1,\ldots,\lambda_{t},1^{s-z};a_2,1^{z-a_2}).
    \end{align*}
    It is straightforward to see that this case is separate from both Case \ref{Case 1} and Case \ref{Case 2} (because of the part $a_2$), and is injective.
\end{case}

\begin{case}[$y\geqslant1 \text{ and } \lambda_{i-1}=\lambda_i=a_3$]\label{Case 4}
    Now, we set 
    \begin{align*}
        f(\lambda)=(\lambda_1,\ldots,\lambda_{i-2},1^{y+a_3};\lambda_{i+1},\ldots,\lambda_t,a_2^{\floor{\frac{x+s}{a_2}}},1^{x+s-a_2\floor{\frac{x+s}{a_2}}}).
    \end{align*}
    At first, let us notice that $\lambda_{i-1}$ exists as a consequence of $y\leqslant a_3-1<w$. Further, we suppose that 
\begin{align*}
    f(\lambda)&=(\lambda_1,\ldots,\lambda_{i(\lambda)-2},1^{y+a_3};\lambda_{i(\lambda)+1},\ldots,\lambda_t,a_2^{\floor{\frac{x+s}{a_2}}},1^{x+s-a_2\floor{\frac{x+s}{a_2}}})\\
    &=(\mu_1,\ldots,\mu_{i(\mu)-2},1^{\Tilde{y}+a_3};\mu_{i(\mu)+1},\ldots,\mu_{\Tilde{t}},a_2^{\floor{\frac{\Tilde{x}+\Tilde{s}}{a_2}}},1^{\Tilde{x}+\Tilde{s}-a_2\floor{\frac{\Tilde{x}+\Tilde{s}}{a_2}}})=f(\mu)
\end{align*}
for some partitions $\mu$ and $\lambda$, where a tilde indicates corresponding parameters for $\mu$. It follows that $i=i(\lambda)=i(\mu)$, $t=\Tilde{t}$ and $y=\Tilde{y}$. Since $y=\Tilde{y}$ and $\lambda_i=\mu_i=a_3$, we also have that $x=\Tilde{x}$. Hence, we must show $s=\Tilde{s}$. But the equality $f(\lambda)=f(\mu)$ asserts that
\begin{align*}
    \floor{\frac{x+s}{a_2}}=\floor{\frac{\Tilde{x}+\Tilde{s}}{a_2}}\hspace{0.2cm}\text{and}\hspace{0.2cm}x+s+a_2\floor{\frac{x+s}{a_2}}=\Tilde{x}+\Tilde{s}+a_2\floor{\frac{\Tilde{x}+\Tilde{s}}{a_2}},
\end{align*}
and thus the required property follows.

It is clear that Case \ref{Case 4} is separate from Case \ref{Case 1}. To see that for other cases, let us notice that $z\geqslant2a_2$. Hence, $1$ appears at least $a_2$ times as a part of $z$ in $f(\lambda)$ in both Case \ref{Case 2} and Case \ref{Case 3}, what is impossible here.
\end{case}

\begin{case}[$y\geqslant1 \text{ and } \lambda_{i-1}>\lambda_i=a_3$]\label{Case 5}
    Here, we define $f(\lambda)$ as
\begin{align*}
    f(\lambda)=(\lambda_1,\ldots,\lambda_{i-1},1^y;\lambda_{i+1},\ldots,\lambda_t,a_2^{\floor{\frac{s}{a_2}}},1^{x+s-a_2\floor{\frac{s}{a_2}}}).
\end{align*}
    By the similar argument as in Case \ref{Case 4}, one can show that if $f(\lambda)=f(\mu)$ in this case, then $\lambda=\mu$. The details are left to the reader. 

    In order to prove that this case is separate from the previous ones, we can notice that $1$ appears at least once and at most $a_3-1$ times as a part of $w$ in $f(\lambda)$ here, what is impossible in the prior cases.
\end{case}

\begin{case}[$y\geqslant 1, \lambda_i>a_3 \text{ and } a_3\nmid y$]\label{Case 6}
    Now, we put
    \begin{align*}
        f(\lambda)=(\lambda_1,\ldots,\lambda_{i-1},a_3^{\floor{\frac{y}{a_3}}},1^{y-a_3\floor{\frac{y}{a_3}}};\lambda_{i+1},\ldots,\lambda_t,a_2^{\floor{\frac{s}{a_2}}},1^{x+s-a_2\floor{\frac{s}{a_2}}}).
    \end{align*}
    At first, let us suppose that $f(\lambda)=f(\mu)$:
    \begin{align*}
        &(\lambda_1,\ldots,\lambda_{i(\lambda)-1},a_3^{\floor{\frac{y}{a_3}}},1^{y-a_3\floor{\frac{y}{a_3}}};\lambda_{i(\lambda)+1},\ldots,\lambda_t,a_2^{\floor{\frac{s}{a_2}}},1^{x+s-a_2\floor{\frac{s}{a_2}}})\\
        =&(\mu_1,\ldots,\mu_{i(\mu)-1},a_3^{\floor{\frac{\Tilde{y}}{a_3}}},1^{\tilde{y}-a_3\floor{\frac{\Tilde{y}}{a_3}}};\mu_{i(\mu)+1},\ldots,\mu_{\Tilde{t}},a_2^{\floor{\frac{\Tilde{s}}{a_2}}},1^{\Tilde{x}+\Tilde{s}-a_2\floor{\frac{\Tilde{s}}{a_2}}}),
    \end{align*}
    where a tilde denotes corresponding values for the partition $\mu$. Since $\lambda_{i(\lambda)},\mu_{i(\mu)}\geqslant a_4$, we need to have that $i=i(\lambda)=i(\mu)$, $t=\Tilde{t}$ and $\lambda_j=\mu_j$ for all $1\leqslant j\leqslant t$ apart (possibly) $j=i$. Moreover, it also follows that $y=\Tilde{y}$, because of
    \begin{align*}
        \floor{\frac{y}{a_3}}=\floor{\frac{\Tilde{y}}{a_3}}\hspace{0.2cm}\text{and}\hspace{0.2cm} y-a_3\floor{\frac{y}{a_3}}=\Tilde{y}-a_3\floor{\frac{\Tilde{y}}{a_3}}.
    \end{align*}
    Further, the equality $f(\lambda)=f(\mu)$ asserts that the following
    \begin{align*}
        s=ka_2+j_1,\hspace{0.2cm} \Tilde{s}=ka_2+j_2\hspace{0.2cm} \text{and}\hspace{0.2cm}x+j_1=\Tilde{x}+j_2
    \end{align*}
    are true for some non-negative integers $j_1,j_2$ and $k$ with $0\leqslant j_1,j_2\leqslant a_2-1$. Without loss of generality, we may assume that $\lambda_i=a_l\geqslant a_m=\mu_i\geqslant a_4$. Thus, it implies the followings
    \begin{align*}
        \lambda_i+j_1=x+y+j_1=\Tilde{x}+y+j_2=\mu_i+j_2 \hspace{0.2cm}\text{and}\hspace{0.2cm} \lambda_i-\mu_i=j_2-j_1<a_2<a_3.
    \end{align*}
    By the assumptions from the statement, we conclude that $x=\Tilde{x}$ and $s=\Tilde{s}$, as required.  

    At this point, we need to show that Case \ref{Case 6} is separate from all of the aforementioned cases. Using the same argument as in Case \ref{Case 5}, one can easily see that Case \ref{Case 6} is separate from cases \ref{Case 1}--\ref{Case 4}. To see that it has no element in common with Case \ref{Case 5}, it is enough to repeat the argument from the previous paragraph. 
\end{case}

\begin{case}[$y\geqslant 1, \lambda_i>a_3 \text{ and } a_3\mid y$]\label{Case 7}
    Here, we set
    \begin{align*}
        f(\lambda)=(\lambda_1,\ldots,\lambda_{i-1},a_3^{\frac{y}{a_3}-1},1^{a_3};\lambda_{i+1},\ldots,\lambda_t,a_2^{\floor{\frac{s}{a_2}}},1^{x+s-a_2\floor{\frac{s}{a_2}}}).
    \end{align*}
    To show that $f(\lambda)=f(\mu)$ implies $\lambda=\mu$ in this case, it is enough to exhibit an analogous reasoning to that one from Case \ref{Case 6}.

 Case \ref{Case 7} is also separate from the others. Observe that $1$ appears exactly $a_3$ times as a part of $w$ in $f(\lambda)$ here, what is impossible elsewhere.
\end{case}

Finally, we conclude that
 \begin{align*}
        p_A(w|\text{ no }a_2\text{'s})p_A(z)\geqslant p_A(w+z|\text{ no } a_2\text{'s}),
    \end{align*}
    as required.
\end{proof}

\begin{re}{\rm
For $A=\{a_i: i\in\mathbb{N}_+\}$ with $a_1<a_{2}<a_3<\cdots$, it seems that one needs to demand that $a_1=1$ in Theorem \ref{Theorem: Injection}. Nevertheless, let us suppose that $a_1>1$, $a_1|a_2$ and $a_1|a_3$. Dividing $A$ by $a_1$ leads us to $\Tilde{A}=\{\Tilde{a}_i:i\in\mathbb{N}_+\}=\{a_i/a_1:i\in\mathbb{N}_+\}$, where $\Tilde{a}_1=1$ and $\Tilde{a}_2,\Tilde{a}_3\in\mathbb{N}_{\geqslant2}$. Since we only demand that $a_1, a_2$ and $a_3$ are integers in the proof of Theorem \ref{Theorem: Injection}, we may apply that result to the set $\Tilde{A}$ if the condition $\Tilde{a}_k-\Tilde{a}_l\geqslant\Tilde{a}_3$ holds for every $k>l\geqslant3$.
 }   
\end{re}

At the first glance, it may be not clear how one can apply Theorem \ref{Theorem: Injection} to obtain some result concerning the Bessenrodt-Ono type inequality. However, the idea is uncomplicated, as the succeeding proof exhibits. 

\iffalse
Indeed, let us assume that $A$ is an arbitrary set of positive integers. For any numbers $a$ and $b$, we may write
\begin{align*}
    p_A(a+b)=p_A(a+b-a_2)+p_A(a+b|\text{ no } a_2\text{'s}).
\end{align*}
Therefore, if $p_A(n)$ satisfies the Bessenrodt-Ono inequality for some initial values $a$ and $b$ and $A$ fulfills the suitable conditions from Theorem \ref{Theorem: Injection}, then one can deduce the following:
\begin{align*}
    p_A(a+b)&=p_A(a+b-a_2)+p_A(a+b|\text{ no } a_2\text{'s})\\
    &<p_A(a-a_2)p_A(b)+p_A(a+b|\text{ no } a_2\text{'s})\\
    &\leqslant p_A(a-a_2)p_A(b)+ p_A(a|\text{ no }a_2\text{'s})p_A(b)\\
    &=(p_A(a-a_2)+p_A(a|\text{ no }a_2\text{'s}))p_A(b)=p_A(a)p_A(b)
\end{align*}
for all sufficiently large parameters $a$ and $b$.

We describe the above procedure in practice in the forthcoming sections.
\fi

\begin{proof}[Proof of Theorem \ref{Theorem: General B-O}]
    Let $w,z\geqslant a_3$. We prove the statement by induction on $w+z$. The correctness of the claim for $2a_3\leqslant w+z\leqslant 2(a_3+a_2)-1$ follows directly from the assumptions. Thus, let us suppose that it is true for every $2a_3\leqslant w+z\leqslant s$ for some $s\geqslant2(a_3+a_2)-1$, and examine the validity of the statement for $w+z=s+1$. Without loss of generality we may require that $w\geqslant z\geqslant a_3$. Since $w+z\geqslant2(a_3+a_2)$ and $w\geqslant z$, we have that $w\geqslant a_3+a_2$. Therefore, it follows that $w-a_2\geqslant a_3$ and
    \begin{align*}
    p_A(w+z)&=p_A(w+z-a_2)+p_A(w+z|\text{ no } a_2\text{'s})\\
    &<p_A(w-a_2)p_A(z)+p_A(w+z|\text{ no } a_2\text{'s})\\
    &\leqslant p_A(w-a_2)p_A(z)+ p_A(w|\text{ no }a_2\text{'s})p_A(z)\\
    &=(p_A(w-a_2)+p_A(w|\text{ no }a_2\text{'s}))p_A(z)=p_A(w)p_A(z),
\end{align*}
where the first inequality is a consequence of the induction hypothesis while the second one is a direct application of Theorem \ref{Theorem: Injection} and the assumption that $2a_2\leqslant a_3$. This completes the proof by the law of induction.
\end{proof}

In the forthcoming sections, we illustrate that both Theorem \ref{Theorem: General B-O} and Theorem \ref{Theorem: Injection} might be applied for various $A$-partition functions. However, it is visible and quite unfortunate that we can not use them to solve the Bessenrodt-Ono inequality for the classical partition function $p(n)$.

\section{The Bessenrodt-Ono inequality for the $m$-ary partition function}

Throughout this section, we assume that $A=\{m^i:i\in\mathbb{N}\}$, where $m\geqslant2$ is an integer. In other words, we consider the well-known $m$-ary partition function $b_m(n)=p_A(n)$ defined in Introduction. Our first aim is to examine the Bessenrodt-Ono type inequality for $b_m(n)$ of the form
\begin{align*}
    b_m(w)b_m(z)>b_m(w+z),
\end{align*}
where $w$ and $z$ are arbitrary positive integers. We solve this problem in a few steps depending on these two parameters.

At first, let us require that $w<m$ and $z=km+l$ for some non-negative integers $k$ and $l$ with $0\leqslant l \leqslant m-1$. Since $b_m(xm+y)=b_m(xm)$ for every non-negative integer $x$ and $0\leqslant y\leqslant m-1$, it follows that
\begin{align*}
    b_m(w)b_m(z)=b_m(w)b_m(km)\leqslant b_m(km+l+w)=b_m(w+z)
\end{align*}
with the equality whenever $0\leqslant w+l \leqslant m-1$.

Now, let us suppose that $m\leqslant w\leqslant 2m-1$. In this case, we have the following.
\begin{lm}\label{Lemma 3.1 w=2m+i z>5m^2}
    Let $m\leqslant w\leqslant 2m-1$ and $z=km+l$ for some positive integer $k$ and $0\leqslant l \leqslant m-1$. If $k\geqslant5m$, then
    \begin{align*}
        b_m(w)b_m(z)>b_m(w+z).
    \end{align*}
\end{lm}
\begin{proof}
    Let us assume that $k,l,w$ and $z$ are as in the statement. In this proof, we use two well-known identities:
    \begin{align}\label{Identity (1) b_m(n)}
        b_m(xm+y)=b_m(xm)
    \end{align}
    and
    \begin{align}\label{Identity (2) b_m(n)}
        b_m(xm)=b_m((x-1)m)+b_m(x),
    \end{align}
    which are valid for all positive integers $x$ and $y$ with $0\leqslant y <m$. Thus, it turns out that
    \begin{align*}
        b_m(w+z)=b_m(km+l+w)\leqslant b_m((k+2)m)=b_m(k+2)+b_m(k+1)+b_m(km)
    \end{align*}
    and
    \begin{align*}
        b_m(w)b_m(z)=2b_m(km).
    \end{align*}
    Hence, it is enough to solve the inequality
    \begin{align*}
        b_m(z)=b_m(km)>b_m(k+2)+b_m(k+1).
    \end{align*}
    Let us assume that $k=dm+j$ for some non-negative integers $d$ and $j$ such that $d\geqslant5$ and $0\leqslant j<m$. Applying both (\ref{Identity (1) b_m(n)}) and (\ref{Identity (2) b_m(n)}), we get that
    \begin{align*}
        b_m(z)=b_m(k)+b_m(k-1)+b_m(k-2)+b_m(k-3)+b_m(k-4)+b_m((k-5)m)
    \end{align*}
    and
    \begin{align*}
        b_m(k+2)+b_m(k+1)&\leqslant 2b_m(k+1)+b_m(d+1)\leqslant2b_m(k)+3b_m(d+1)\\
        &\leqslant b_m(k)+b_m(k-1)+b_m(d)+3b_m(d+1).
    \end{align*} 
    In consequence, it suffices to show that
    \begin{align*}
        b_m(k-2)+b_m(k-3)+b_m(k-4)+b_m((k-5)m)>b_m(d)+3b_m(d+1)
    \end{align*}
    Since $b_m((k-5)m)>b_m(k-5)$, it is enough to examine the validity of $k-5\geqslant d$, but it is an easy exercise.
\end{proof}

Lemma \ref{Lemma 3.1 w=2m+i z>5m^2} allows us to restrict our consideration to the case when $m\leqslant z< 5m^2$. Since we still have infinitely many possibilities to choose $m$ and $z$, an additional auxiliary lemma is needed.

\begin{lm}\label{Lemma 3.2}
    Let $m\geqslant4$, $m\leqslant w\leqslant 2m-1$ and $z=km+l$ for some integers $k$ and $l$ with $1\leqslant k\leqslant 5m-1$ and $0\leqslant l<m$.  For $k=1$, we have that
    \begin{align*}
        b_m(w)b_m(z)\geqslant b_m(w+z)
    \end{align*}
    with the equality whenever $w+z\geqslant3m$.
    
    If $k\geqslant2$ and $m\geqslant6$, then 
    \begin{align*}
        b_m(w)b_m(z)>b_m(w+z).
    \end{align*}
\end{lm}
\begin{proof}
    Let us assume that all the parameters are as in the claim. One can easily check that the statement is true for $k=1,2$. Therefore, let us demand that $k\geqslant3$. It follows that
    \begin{align*}
        b_m(w)b_m(z)=2b_m(km)
    \end{align*}
    and
    \begin{align*}
        b_m(w+z)\leqslant b_m((k+2)m)=b_m(k+2)+b_m(k+1)+b_m(km).
    \end{align*}
    Since $m\geqslant6$, we obtain that $k+1<k+2\leqslant 5m+1<m^2$. Hence, it is not difficult to observe that
    \begin{align*}
        b_m(k+2)+b_m(k+1)<\frac{k+2}{m}+1+\frac{k+1}{m}+1 \hspace{0.2cm}\text{and}\hspace{0.2cm} b_m(km)\geqslant k+1.
    \end{align*}
    Thus, it suffices to show that the following
    \begin{align*}
        k+1>\frac{k+2}{m}+\frac{k+1}{m}+2 
    \end{align*}
    is true, which is an elementary exercise.
\end{proof}

Now, we can check all the exceptions in the Bessenrodt-Ono inequality for $m\leqslant 5$, $m\leqslant w<2m$ and $m\leqslant z<5m^2$ one by one. However, we omit that task here, because the complete list of the exceptions will be given later (see, Table 1). In the meantime, we proceed to the next case. 

At this point, we require that both $w,z\geqslant2m$. In such a setting the following property is true.

\begin{lm}\label{Lemma 3.3 (w,z>2m)}
    Let $m\geqslant4$ be fixed. For every $w,z\geqslant 2m$ with $w+z\leqslant 2m^2$, we have
    \begin{align*}
        b_m(w)b_m(z)>b_m(w+z).
    \end{align*}
\end{lm}
\begin{proof}
    Let us write $w=cm+j$ and $z=dm+i$ for some positive integers $d\geqslant c\geqslant2$ and $0\leqslant i,j<m$. At first, we consider the case when $c=d=2$. Since $m\geqslant4$, it follows that
    \begin{align*}
        b_m(2m+i)b_m(2m+j)=9
    \end{align*}
    and
    \begin{align*}
        b_m(4m+i+j)\leqslant b_m(5m)\leqslant b_4(20)=8.
    \end{align*}
    Therefore, we can assume that $d\geqslant 3$. Let us deal with the case when $w+z<2m^2$. We have that $c\leqslant m-1$, and it is clear that 
    \begin{align*}
        b_m(w)b_m(z)=(c+1)b_m(dm).
    \end{align*}
    On the other hand, the identities \eqref{Identity (1) b_m(n)} and \eqref{Identity (2) b_m(n)} assert that
    \begin{align*}
        b_m(w+z)\leqslant b_m((c+d+1)m)=b_m(d+c+1)+\cdots+b_m(d+1)+b_m(dm)
    \end{align*}
    for $c+d<2m-1$. Analogously, we have that
    \begin{align*}
        b_m(w+z)\leqslant b_m((c+d)m)=b_m(d+c)+\cdots+b_m(d+1)+b_m(dm)
    \end{align*}
    for $c+d=2m-1$. Hence, it is transparent that the following inequalities
\begin{align*}
    cb_m(dm)\geqslant c(d+1)\geqslant4c>2(c+1),
\end{align*}
    hold, as required.

    Next, let us suppose that $w+z=2m^2$. If $w=z=m^2$, then
    \begin{align*}
        b_m(m^2)b_m(m^2)=(m+2)^2>3m+3=b_m(2m^2).
    \end{align*}
    Therefore, we consider the case when $w<m^2$ and $z>m^2$. We get that
    \begin{align*}
        b_m(cm+j)b_m(dm+i)=(c+1)(2d-m+2)\geqslant(c+1)(m+2),
    \end{align*}
    where the equality follows from the fact that we can either take $m^2$ as a part of $z$ or not. Finally, we have
    \begin{align*}
        b_m(w)b_m(z)\geqslant(c+1)(m+2)\geqslant3(m+2)>3m+3=b_m(2m^2).
    \end{align*}
    This completes the proof.
\end{proof}

Now, let us exhibit all the exceptions in the Bessenrodt-Ono inequality for $b_m(n)$ which are not covered by Lemmas \ref{Lemma 3.1 w=2m+i z>5m^2}--\ref{Lemma 3.3 (w,z>2m)}. We gather them all in Table 1.

\begin{center}
\begin{table}
\begin{tabular}{ |c|c| } 
\hline
$m$ & Pairs $(w,z)$ with $m\leqslant w\leqslant z$ and $b_m(w)b_m(z)\leqslant b_m(w+z)$\\
 \hline\hline 
 $2$ & \textcolor{blue}{$(2,2)$}, \textcolor{blue}{$(2,3)$}, $(3,3)$, $(3,5)$, $(3,7)$, \textcolor{blue}{$(3,9)$} \\ \hline
 $3$ & $(4,5)$, $(4,8)$, $(5,5)$, $(5,7)$, $(5,8)$, \textcolor{blue}{$(7,8)$}, \textcolor{blue}{$(8,8)$}  \\ \hline
 $4$ & \textcolor{blue}{$(5,11)$}, \textcolor{blue}{$(5,15)$}, \textcolor{blue}{$(6,10)$}, \textcolor{blue}{$(6,11)$}, \textcolor{blue}{$(6,14)$}, \textcolor{blue}{$(6,15)$},\\
  \phantom{} & \textcolor{blue}{$(7,9)$}, \textcolor{blue}{$(7,10)$}, \textcolor{blue}{$(7,11)$}, \textcolor{blue}{$(7,13)$}, \textcolor{blue}{$(7,14)$}, \textcolor{blue}{$(7,15)$}   \\
 \hline
\end{tabular}
\caption{All the pairs $(w,z)$ with $m\leqslant w\leqslant z$ such that $b_m(w)b_m(z)\leqslant b_m(w+z)$, which are not covered by Lemmas \ref{Lemma 3.1 w=2m+i z>5m^2}--\ref{Lemma 3.3 (w,z>2m)}. The blue color indicates that both sides of the inequality are equal.}
\end{table}
\end{center}

After all of the discussion above, we are finally ready to completely solve the Bessenrodt-Ono inequality for the $m$-ary partition function.

\begin{thm}\label{Theorem: B-O b_m(n)}
    Let $m\geqslant2$ be fixed. For all positive integers $w,z\geqslant m$ such that $w+z\geqslant n_m$, we have that
    \begin{align*}
        b_m(w)b_m(z)>b_m(w+z),
    \end{align*}
    where the values $n_m$ are collected in Table 2.
    \begin{center}
\begin{table}
\begin{tabular}{ |c|c|c|c|c| } 
\hline
$m$ & $2$ & $3$ & $4$ & $\geqslant5$\\
 \hline
 $n_m$ & $13$ & $17$ & $23$ & $4m-1$ \\ \hline
\end{tabular}
\caption{The values of $n_m$ which imply that the inequality $b_m(w)b_m(z)>b_m(w+z)$ holds for all $w,z\geqslant m$ with $w+z\geqslant n_m$.}
\end{table}
\end{center}
\end{thm}

\begin{proof}
Let us assume that $m\geqslant5$. If at least one of the parameter $w$ or $z$ is smaller than $2m$, then the required property follows from Lemma \ref{Lemma 3.1 w=2m+i z>5m^2} and Lemma \ref{Lemma 3.2} (and some numerical computations for $m=5$). If we have that $w,z\geqslant2m$ with $w+z\leqslant2m^2$ then the required property is a consequence of Lemma \ref{Lemma 3.3 (w,z>2m)}. Let us further proceed by the induction on $w+z$. Without loss of generality we demand that $w\geqslant z$. Since the appropriate inequality holds for every $4m\leqslant w+z\leqslant N$ for some $N\geqslant 2m^2$, let us check the correctness of the statement for $w+z=N+1$. We have that $w\geqslant m^2+1$ and
\begin{align*}
    b_m(w+z)&=b_m(w+z-m)+b_m(w+z|\text{ no }m\text{'s})\\
    &<b_m(w-m)b_m(z)+b_m(w|\text{ no }m\text{'s})b_m(z)\\
    &=(b_m(w-m)+b_m(w|\text{ no }m\text{'s}))b_m(z)=b_m(w)b_m(z),
\end{align*}
where the inequality is a consequence of both the induction hypothesis (observe that $w+z-m>2m^2-m>4m$ and $w-m\geqslant m^2+1-m>2m$) and Theorem \ref{Theorem: Injection}.

Since the proofs for $m=2,3,4$ are very similar, we omit them here. 
\end{proof}

We see that Theorem \ref{Theorem: Injection} plays a crucial role in the above proof. However, it is worth noting that we do not need that property to deal with the Bessenrodt-Ono inequality for the $m$-ary partition function. In fact, it is suffices to systematically apply both \eqref{Identity (1) b_m(n)} and \eqref{Identity (2) b_m(n)}, as follows.

\begin{proof*}{Sketch of the second proof of Theorem~\ref{Theorem: B-O b_m(n)}}
    Let $m\geqslant 2$, and let us assume that $w\leqslant z$, $w=cm+j$ and $z=dm+i$ for some non-negative integers $c,d,i,j$ with $0\leqslant i,j<m.$
    Using both \eqref{Identity (1) b_m(n)} and \eqref{Identity (2) b_m(n)}, we obtain that
    \begin{align*}
        b_m(w)b_m(z)=b_m(cm)b_m(dm)
    \end{align*}
    and
    \begin{align*}
        b_m(w+z)\leqslant b_m((c+d+1)m)&=b_m(c+d+1)+\cdots+b_m(d+1)+b_m(dm)\\
        &\leqslant (c+1)b_m(2d+1)+b_m(dm).
    \end{align*}
    Hence, it is enough to examine the validity of the following inequality
    \begin{align*}
        (b_m(cm)-1)b_m(dm)>(c+1)b_m(2d+1).
    \end{align*}
    One can simplify it to 
    \begin{align*}
        \begin{cases}
            b_m(cm)-1&>c+1\\
            b_m(dm)&\geqslant b_m(2d+1).
        \end{cases}
    \end{align*}
    Hence, if we demand that $c\geqslant m+1$, then it is not difficult to see that $b_m(cm)\geqslant c+3$. That is a consequence of the fact that we may take $m^2$ as a part. If we do that, then the assumption $c\geqslant m+1$ implies that $m$ might occur as a part at least once. Thus, we get two additional partitions at worst.

    Now, let us observe that for $m\geqslant3$ the inequality $b_m(dm)\geqslant b_m(2d+1)$ is automatically satisfied. For $m=2$, we can just write $b_2(2d)=b_2(2d+1)$.

     Therefore, we have already established that the inequality $b_m(w)b_m(z)>b_m(w+z)$ holds for all $w,z\geqslant m^2+m$. In order to prove the theorem completely, we need to apply some analogues of Lemmas \ref{Lemma 3.1 w=2m+i z>5m^2}--\ref{Lemma 3.3 (w,z>2m)}, which cover all of the remaining possibilities for choosing $w$ and $z$.    
\end{proof*}

We end this section with a direct application of Theorem \ref{Theorem: B-O b_m(n)} to the extended $m$-ary partition function.

\begin{thm}\label{Theorem: Max b_m(n)}
    Let $m\geqslant2$. If $m\geqslant3$, then the maximal value $\max{b_m(n)}$ of the $m$-ary partition function on $B_m(n)$ is attained at the partition
    \begin{align*}
        &\left(m^{\floor{\frac{n}{m}}},1^{n-m\floor{\frac{n}{m}}}\right).
    \end{align*}
    For $m=2$, it is attained at the partitions of the form
    \begin{align*}
        \left(4^i,2^{\floor{\frac{n}{2}}-2i},1^{n-2\flce{\frac{n}{2}}}\right),
    \end{align*}
    where $0\leqslant i \leqslant \floor{\frac{n}{4}}$. 
    
    In particular, we have
    \begin{align*}
        \max{b_m(n)}=2^{\floor{\frac{n}{m}}}
    \end{align*}
    for every $m\geqslant2$.
\end{thm}
\begin{proof}
    At first, let us notice that $b_m(m)=2$, $b_m(m^2)=m+2$ and $b_m(m^3)=(m^2+4)(m+1)/2-m$. Thus, it is clear that $b_2(2)b_2(2)=b_2(4)$, $b_2(4)b_2(4)>b_2(8)$ and $b_m^m(m^{i-1})>b_m(m^i)$ for $m\geqslant3$ and $i=2,3$. Next, using Theorem \ref{Theorem: B-O b_m(n)} we may write
    \begin{align*}
        b_m(m^k)&=b_m((m-1)m^{k-1}+m^{k-1})<b_m((m-1)m^{k-1})b_m(m^{k-1})\\
        &=b_m((m-2)m^{k-1}+m^{k-1})b_m(m^{k-1})<b_m((m-2)m^{k-1})b_m^2(m^{k-1})\\
        &=\cdots\\
        &=b_m([m-(m-1)]m^{k-1}+m^{k-1})b_m^{m-2}(m^{k-1})<b_m^m(m^{k-1})
    \end{align*}
    for every $k\geqslant4$. Therefore, it follows from the induction that $b_m^m(m^{k-1})>b_m(m^k)$ for all $m\geqslant2$ and $k\geqslant3$. Now, the first part of the statement is clear. The second one, on the other hand, is a straightforward consequence of elementary computations.
\end{proof}

\section{The Bessenrodt-Ono inequality for the power partition function}

In this section, we deal with the Bessenrodt-Ono inequality for the ($d$-th) power partition function $p_d(n):=p_{A_d}(n)$, where $A_d=\{n^d: n\in\mathbb{N}_+\}$ and $d\in\mathbb{N}_{\geqslant 2}$. 

At the beginning, let us observe that $4^2-3^2<3^2$. Thus, we can not simply apply Theorem \ref{Theorem: General B-O} (or Theorem \ref{Theorem: Injection}) to attack the problem in general. On the other hand, if $d\geqslant3$ and $k>l\geqslant3$, then $k^d-l^d>(k-l)dl^{d-1}\geqslant3^d.$ Therefore, it follows that Theorem \ref{Theorem: General B-O} might be used when $d\geqslant3$. For the sake of completeness, we investigate the Bessenrtodt-Ono inequality for $p_{A_2}(n)$ separately by applying the following analogue of Theorem \ref{Theorem: Injection}.
\begin{pr}\label{Theorem: Injection (2)}
    Let $A=\{1,a_2,a_3,\ldots\}$ be an arbitrary (finite or infinite) set of positive integers with $1<a_2<a_3<\cdots$ such that $a_k-a_l\geqslant a_3$ for every $k>l\geqslant4$. Then, the inequality  
    \begin{align*}
        p_A(w|\text{ no }a_2\text{'s})p_A(z)\geqslant p_A(w+z|\text{ no } a_2\text{'s})
    \end{align*}
    holds for all positive integers $w\geqslant a_3+1$ and $z\geqslant 3a_2$.
\end{pr}
\begin{proof}
    Let $A,w,z$ be as in the statement. Similarly to the proof of Theorem \ref{Theorem: Injection}, we construct an injective map 
    \begin{align*}
        g:P_A(w+z|\text{ no }a_2\text{'s})\to P_A(w|\text{ no }a_2\text{'s})\oplus P_A(z)
    \end{align*}
by setting

    \begin{align*}
        g(\lambda)=\begin{cases}
            (\lambda_1,\lambda_2,\ldots,\lambda_{t-1},1^{s-z+\lambda_t};a_2^2,1^{z-2a_2}), & \text{ if } y=0, z<s\leqslant z+a_3, \lambda_t\geqslant a_4,\\
            f(\lambda), & \text{ otherwise,}
        \end{cases}
    \end{align*}
    where all the notation comes from the proof of Theorem \ref{Theorem: Injection}. 
    
    It is not difficult to show that the map $g(\lambda)$ is well-defined and injective. 
    The details are left to the reader.
\end{proof}

Now, we are ready to prove the main theorem of this section. 

\begin{thm}\label{Theorem: B-O p_A_d(n)}
    Let $d\geqslant3$ be fixed. For all positive integers $w,z\geqslant 3^d$, we have that
    \begin{align*}
        p_{A_d}(w)p_{A_d}(z)>p_{A_d}(w+z).
    \end{align*}
    If $d=2$, then the above is valid for every $w,z\geqslant12$.
\end{thm}

\begin{proof}
Let us assume that $d\geqslant3$. At first, we demand that $w,z\geqslant3^d$ and $2\cdot3^d\leqslant w+z\leqslant2(3^d+2^d)-1$. We have that
\begin{align*}
    p_{A_d}(w)p_{A_d}(z)\geqslant \left(1+\frac{w}{2^d}\right)\left(1+\frac{z}{2^d}\right)\geqslant\left(1+\frac{3^d}{2^d}\right)^2.
\end{align*}
On the other hand, one can derive that
\begin{align*}
    p_{A_d}(w+z)\leqslant p_{A_d}(2(3^d+2^d))\leqslant6+\frac{3^d+2^{d+1}}{2^d}+\frac{3^d+2^{d}}{2^{d-1}}=10+\frac{3^d}{2^d}+\frac{3^{d}}{2^{d-1}},
\end{align*}
where the last inequality follows from elementary computations and the observation that $4^d$ might occur as a part of $2(3^d+2^d)$  (for $d=3$). Hence, it suffices to check the validity of 
\begin{align*}
    \left(1+\frac{3^d}{2^d}\right)^2>10+\frac{3^d}{2^d}+\frac{3^{d}}{2^{d-1}},
\end{align*}
or equivalently of
\begin{align*}
    u^2-u-9>0,
\end{align*}
where $u=(3/2)^d$. It is easy to solve the above problem and verify that the required property holds whenever $d\geqslant4$. For $d=3$, one can check one by one that 
\begin{align*}
    p_{A_3}(w)p_{A_3}(z)>p_{A_3}(w+z)
\end{align*}
is satisfied for all $2\cdot3^3\leqslant w+z\leqslant 2(3^3+2^3)-1$, as required. 
Therefore, Theorem \ref{Theorem: General B-O} completes the proof for every $d\geqslant3$.

In order to show the property for $d=2$, one can just exhibit a similar reasoning to the above one and apply Theorem \ref{Theorem: Injection (2)} instead of Theorem \ref{Theorem: General B-O}. The details are left to the reader.
\end{proof}

As a consequence of the Bessenrodt-Ono inequality for the power partition function $p_{A_d}(n)$, we get an analogue of Theorem \ref{Theorem: Max b_m(n)}.

\begin{thm}\label{Theorem: Max p_{A_d}(n)}
    Let $d\geqslant2$. If $d\geqslant3$, then the maximal value $\max{p_{A_d}(n)}$ of the power partition function on $P_{A_d}(n)$ is attained at the partition
    \begin{align*}
        &\left((2^{d})^{\floor{\frac{n}{2^d}}},1^{n-2^d\floor{\frac{n}{2^d}}}\right).
    \end{align*}
    For $d=2$, it is attained at the partitions of the form
    \begin{align*}
        (4^{\floor{\frac{n}{4}}},1^{n-4\floor{\frac{n}{4}}}), \hspace{0.2cm} &\text{if} \hspace{0.2cm} n\equiv 0,1,2,3\pmod*{4},\\
        (9,4^{\floor{\frac{n-9}{4}}},1^{n-4\floor{\frac{n-9}{4}}-9}), \hspace{0.2cm} &\text{if} \hspace{0.2cm} n\equiv1,2,3\pmod*{4} \text{ and }n\geqslant9,\\
        (9^{2},4^{\floor{\frac{n-18}{4}}},1^{n-4\floor{\frac{n-18}{4}}-18}), \hspace{0.2cm} &\text{if} \hspace{0.2cm} n\equiv 2,3\pmod*{4} \text{ and }n\geqslant18,\\
        (9^3,4^{\floor{\frac{n-27}{4}}},1^{n-4\floor{\frac{n-27}{4}}-27}), \hspace{0.2cm} &\text{if} \hspace{0.2cm} n\equiv 3\pmod*{4} \text{ and }n\geqslant27.
        %(4^{\floor{\frac{n}{4}}},1^{n-4\floor{\frac{n}{4}}})\text{ and } (9,4^{\floor{\frac{n-9}{4}}},1^{n-4\floor{\frac{n-9}{4}}-9}), \hspace{0.2cm} &\text{if} \hspace{0.2cm} n\not\equiv 0\pmod*{4} \text{ and }n\geqslant9.
    \end{align*}
    
    In particular, we have
    \begin{align*}
        \max{p_{A_d}(n)}=2^{\floor{\frac{n}{2^d}}}
    \end{align*}
    for every $d\geqslant2$.
\end{thm}
\begin{proof}
    At first, let us notice that $p_{A_d}(2^d)=2$ and $p_{A_d}(3^d)=\floor{\frac{3^d}{2^d}}+2$.
    Thus, if we consider the extended power partition function and have a partition $\lambda$ with\linebreak a part, let say, $\lambda_i=3^d$, then we may replace that part by $\floor{\frac{3^d}{2^d}}$ parts equal to $2^d$ and $3^d-2^d\cdot\floor{\frac{3^d}{2^d}}$ ones. It is not difficult to verify that $2^{\floor{\frac{3^d}{2^d}}}\geqslant\floor{\frac{3^d}{2^d}}+2$ for every $d\geqslant2$ (with the equality if and only if $d=2$). Therefore, we see that after such\linebreak a transformation the value of the extended power partition function grows for $d\geqslant3$ or remains the same if $d=2$. 

    Now, let us demand that $d\geqslant3$ and $n\geqslant4$. Theorem \ref{Theorem: B-O p_A_d(n)} asserts that
    \begin{align*}
        p_{A_d}(n^d)&=p_{A_d}(n^d-3^d+3^d)<p_{A_d}(n^d-3^d)p_{A_d}(3^d)\\
        &=p_{A_d}(n^d-2\cdot3^d+3^d)p_{A_d}(3^d)<p_{A_d}(n^d-2\cdot3^d)p_{A_d}^2(3^d)\\
        &=\cdots\\
        &=p_{A_d}\left(n^d-3^d\left(\floor{\frac{n^d}{3^d}}-1\right)+3^d\right)p_{A_d}^{\floor{\frac{n^d}{3^d}}-2}(3^d)\\
        &<p_{A_d}\left(n^d-3^d\left(\floor{\frac{n^d}{3^d}}-1\right)\right)p_{A_d}^{\floor{\frac{n^d}{3^d}}-1}(3^d).
    \end{align*}
Next, we can put 
\begin{align*}
    w:=n^d-3^d\left(\floor{\frac{n^d}{3^d}}-1\right), 
\end{align*}
and observe that $3^d\leqslant w\leqslant2\cdot3^d-1$. Hence, it turns out that
\begin{align*}
    p_{A_d}(w)=\floor{\frac{w}{2^d}}+\floor{\frac{w-3^d}{2^d}}+2<2\left(\floor{\frac{w}{2^d}}+1\right).
\end{align*}
However, it is easy to check that the inequality
\begin{align*}
    2^{\floor{\frac{w}{2^d}}}\geqslant2\left(\floor{\frac{w}{2^d}}+1\right)
\end{align*}
holds for every $d\geqslant3$. Hence, we get that $p_{A_d}(w)<p_{A_d}^{\floor{\frac{w}{2^d}}}(2^d)$ and
\begin{align*}
    p_{A_d}(n^d)<p_{A_d}^{\floor{\frac{w}{2^d}}}(2^d)p_{A_d}^{\floor{\frac{n^d}{3^d}}-1}(3^d).
\end{align*}
 The commentary from the beginning completes the proof for $d\geqslant3$.

 Finally, we deal with the case when $d=2$. The first few sentences of the proof ensure that $$p_{A_2}(\bm{(4,4)})=p_{A_2}(\bm{(4,4,1)})=p_{A_2}(\bm{(9)})=4.$$
 Furthermore, the similar equality holds if the part $9$ occurs two or three times. More precisely, we have
 \begin{align*}
     p_{A_2}(\bm{(4^4)})=p_{A_2}(\bm{(4^4,1^2)})=p_{A_2}(\bm{(9^2)})=2^4
     \end{align*}
     and
     \begin{align*}
p_{A_2}(\bm{(4^6)})=p_{A_2}(\bm{(4^6,1^3)})=p_{A_2}(\bm{(9^3)})=2^6.
 \end{align*}
 However, the part $9$ can not occur more than three times. Indeed, if there are four nines, then
 \begin{align*}
     p_{A_2}(\bm{(4^9)})=2^9>p_{A_2}(\bm{(9^4)})=2^8.
 \end{align*}
Now, let us suppose that $n=4s$ for some positive integer $s\geqslant3$. It follows that
\begin{align*}
    p_{A_2}(\bm{(4^s)})=2^s>2^{s-1}=p_{A_2}(\bm{(9,4^{s-3},1^3)}).
\end{align*}
 Further, if $n=4s+1$ for some $s\geqslant2$, then
 \begin{align*}
    p_{A_2}(\bm{(4^s,1)})=p_{A_2}(\bm{(9,4^{s-2})})=2^s>2^{s-1}\geqslant p_{A_2}(\bm{(9^2,4^{s-5},1^{3})}).
\end{align*}
One can make similar computations for $n\equiv2,3\pmod*{4}$ to deduce that the partitions from the statement might be the appropriate candidates for the maximal value $\max{p_{A_2}(n)}$ of the ($2$nd) power partition function on $P_{A_2}(n)$.

Next, let us assume that $n$ is arbitrary and $\lambda$ is a partition of $n$ with at least one part equals $16$. Observe that
$p_{A_2}(\bm{(16)})=8<16=p_{A_2}(\bm{(4^4)})$. Hence, if we replace the part $16$ by four parts equal $4$, then the value of the extended power partition function (for $d=2$) becomes larger. Similarly, one can easily check that $p_{A_2}(\bm{(25)})=19<32=p_{A_2}(\bm{(16,4^2,1)})<64=p_{A_2}(\bm{(4^6,1)})$. Thus, let us suppose that $l\geqslant 6$ and $l^2$ occurs as a part of a partition of $n$. Once again, Theorem \ref{Theorem: B-O p_A_d(n)} ensures that 
\begin{align*}
    p_{A_2}(l^2)&=p_{A_2}(l^2-4^2+4^2)<p_{A_2}(l^2-4^2)p_{A_2}(4^2)\\
    &=p_{A_2}(l^2-2\cdot4^2+4^2)p_{A_2}(4^2)<p_{A_2}(l^2-2\cdot4^2)p_{A_2}^2(4^2)\\
    &=\cdots\\
    &=p_{A_2}\left(l^2-4^2\cdot\left(\floor{\frac{l^2}{4^2}}-1\right)+4^2\right)p_{A_2}^{\floor{\frac{l^2}{4^2}}-2}(4^2)\\
    &<p_{A_2}\left(l^2-4^2\cdot\left(\floor{\frac{l^2}{4^2}}-1\right)\right)p_{A_2}^{\floor{\frac{l^2}{4^2}}-1}(4^2).
\end{align*}
Now, one can make some numerical calculations to check the values of $p_{A_2}(n)$ and to verify the validity of the statement for every $1\leqslant n\leqslant31$. If we do so, then it follows that
\begin{align*}
    p_{A_2}(l^2)<p_{A_2}^{\floor{\frac{l^2}{2^2}}}(2^2),
\end{align*}
as required. This completes the proof.
\end{proof}

\section{Concluding Remarks}

At the end of this paper, we present two additional examples in order to illustrate the usefulness of both Theorem \ref{Theorem: General B-O} (\ref{Theorem: Injection}) and Proposition \ref{Theorem: Injection (2)} and encourage the reader to investigate the Bessenrodt-Ono type inequalities for other $A$-partition functions. 

The first instance reefers to the well-known Fibonacci numbers. More precisely, let us set $A=F:=\{1,2,3,5,8,13,\ldots\}=\{F_n:n\geqslant2\}$ and examine the multiplicative property for the Fibonacci partition function $p_F(n)$.
\begin{thm}\label{Theorem Fib}
    Let $F=\{F_n:n\geqslant2\}$. For every $a,b\geqslant6$, we have that
    $$p_F(a)p_F(b)>p_F(a+b).$$
\end{thm}
\begin{proof}
    We prove the statement by induction on $a+b$, where $a,b\geqslant6$. At first, we verify the Bessenrodt-Ono type inequality one by one for all the values $12\leqslant a+b\leqslant15$ with $a,b\geqslant6$. Further, it is convenient to assume that $a\geqslant b$ and that the statement is true for every $a+b=12,\ldots,s-1$ for some $s\geqslant16$. For $a+b=s$, we just have
    \begin{align*}
        p_F(a+b)&=p_F(a+b-2)+p_F(a+b|\text{ no } 2\text{'s})\\
        &<p_F(a-2)p_F(b)+p_F(a|\text{ no } 2\text{'s})p_F(b)=p_F(a)p_F(b),
    \end{align*}
    where the inequality is a consequence of both Proposition \ref{Theorem: Injection (2)} and the induction hypothesis.
\end{proof}

In consequence, we obtain the following. 

\begin{thm}\label{Theorem max p_F}
    The maximal value $\max{p_{F}(n)}$ of the Fibonacci partition function on $P_{F}(n)$ is attained at the partitions of the form
    \begin{align*}
        (3^\frac{n}{3}), \hspace{0.2cm} &\text{if} \hspace{0.2cm} n\equiv 0\pmod*{3},\\
        (3^{\floor{\frac{n}{3}}-1},2^2) \hspace{0.2cm} &\text{if} \hspace{0.2cm} n\equiv1\pmod*{3},\\
        (5,3^{\floor{\frac{n}{3}}-2},2) \hspace{0.2cm} &\text{if} \hspace{0.2cm} n\equiv1\pmod*{3} \text{ and }n\geqslant7,\\
        (5^2,3^{\floor{\frac{n}{3}}-3}) \hspace{0.2cm} &\text{if} \hspace{0.2cm} n\equiv1\pmod*{3} \text{ and }n\geqslant10,\\
        (3^{\floor{\frac{n}{3}}},2) \hspace{0.2cm} &\text{if} \hspace{0.2cm} n\equiv2\pmod*{3},\\
        (5,3^{\floor{\frac{n}{3}}-1}) \hspace{0.2cm} &\text{if} \hspace{0.2cm} n\equiv2\pmod*{3} \text{ and }n\geqslant5.
        %(4^{\floor{\frac{n}{4}}},1^{n-4\floor{\frac{n}{4}}})\text{ and } (9,4^{\floor{\frac{n-9}{4}}},1^{n-4\floor{\frac{n-9}{4}}-9}), \hspace{0.2cm} &\text{if} \hspace{0.2cm} n\not\equiv 0\pmod*{4} \text{ and }n\geqslant9.
    \end{align*}
    
    In particular, we have
    \begin{align*}
        \max{p_{F}(n)}=\begin{cases}
            3^\frac{n}{3}  \hspace{0.2cm} &\text{if} \hspace{0.2cm} n\equiv0\pmod*{3},\\
            2^2\cdot3^{\floor{\frac{n}{3}}-1}  \hspace{0.2cm} &\text{if} \hspace{0.2cm} n\equiv1\pmod*{3},\\
            2\cdot3^{\floor{\frac{n}{3}}}  \hspace{0.2cm} &\text{if} \hspace{0.2cm} n\equiv2\pmod*{3}.
        \end{cases}
    \end{align*}
\end{thm}

\begin{proof}
    At first, we make numerical calculations to check the validity of the statement for $n\leqslant16$. Let us notice that $p_F(5)=p_F(3)p_F(2)=6$. 

    Next, considering all possible residue class of $n\pmod*{3}$, one can easy verify that the values of the extended Fibonacci partition function agrees on the appropriate partitions in the statement.
    
    Let us assume that $n\geqslant17$. If we have a partition of $n$ with three parts equal to $5$, then we can replace them by five threes:
    \begin{align*}
        p_F(\bm{(5^3)})=6^3<3^5=p_F(\bm{(3^5)}).
    \end{align*}
    The analogous argument shows that $2$ also can not appear more than two times as a part of $n$. It is also worth noting that if we consider the maximum value, then neither $8$ nor $13$ may occur as a part of $n$ --- that is because of
    \begin{align*}
        p_F(8)=14<18=\max{p_{F}(8)}=p_F(\bm{(5,3)})=p_F(\bm{(3^2,2)})
    \end{align*}
    and
    \begin{align*}
        p_F(13)=41<108=\max{p_{F}(13)}=p_F(\bm{(5^2,3)})=p_F(\bm{(5,3^2,2)})=p_F(\bm{(3^3,2^2)}).
    \end{align*}
    Suppose further that $F_l$ for some $l\geqslant8$ occurs as a part of $n$. By Theorem \ref{Theorem Fib}, we get that
    \begin{align*}
        p_F(F_l)=p_F(F_l-8+8)&<p_F(F_l-8)p_F(8)<p_F(F_l-2\cdot8)p_F^2(8)\\
                &<\cdots<p_F\left((F_l-8\cdot\left(\floor{\frac{F_l}{8}}-1\right)\right)p_F^{\left(\floor{\frac{F_l}{8}}-1\right)}(8).
    \end{align*}
    After checking the validity of the statement and calculating the values of $p_F(n)$ for $n\leqslant16$, we conclude the required property.
\end{proof}

Our last example of the Bessenrodt-Ono type inequality is related to a slightly more exotic $A$-partition function. Actually, we put $A=N:=\{n!:n\geqslant1\}$ and examine the property for, let say, the `factorial' partition function. Since we have already applied Proposition \ref{Theorem: Injection (2)} to establish Theorem \ref{Theorem Fib}, let us now use Theorem \ref{Theorem: General B-O} to prove the following.

\begin{thm}
    Let $N=\{n!:n\geqslant1\}$. For every $a,b\geqslant6$, we have that
    $$p_N(a)p_N(b)>p_N(a+b).$$
\end{thm}
\begin{proof}
    It is enough the check the correctness of the claim for all $a,b\geqslant6$ with $12\leqslant a+b\leqslant15$. If we do so, then Theorem \ref{Theorem: General B-O} completes the proof.
\end{proof}

As a direct application of the above property, we obtain the following.

\begin{thm}
    The maximal value $\max{p_{N}(n)}$ of the factorial partition function on $P_{N}(n)$ is attained at the partition
    \begin{align*}
        (2^{\floor{\frac{n}{2}}},1^{n-2\floor{\frac{n}{2}}}).
    \end{align*}
    Moreover, we get
    \begin{align*}
        \max{p_{N}(n)}=2^{\floor{\frac{n}{2}}}.
    \end{align*}
\end{thm}
\begin{proof}
    We leave the proof as an exercise for the reader.
\end{proof}

At the end of this manuscript, let us state a conjecture which arises from many works concerning Bessenrodt-Ono type inequalities, especially from our investigation and Gajdzica's recent paper \cite{KG2}.

\begin{con}
    Let $A$ be an arbitrary set of positive integers with $\gcd A=1$ and $\#A\geqslant2$. For all sufficiently large values of $a$ and $b$, we have
    \begin{align*}
        p_A(a)p_A(b)>p_A(a+b).
    \end{align*}
\end{con}

\iffalse
\begin{re}{\rm
    It is worth noting that the correctness of the above conjecture for\linebreak a finite set $A$ was established by Gajdzica \cite[Theorem 5.4]{KG2}.
}
\end{re}
\fi

\section*{Acknowledgments}
I would like to thank Piotr Miska and Maciej Ulas for their time, profound comments and valuable suggestions. This research was funded by both a grant of the National Science Centre (NCN), Poland, no. UMO-2019/34/E/ST1/00094. %and a grant from the Faculty of Mathematics and Computer Science under the Strategic Program Excellence Initiative at the Jagiellonian University in Kraków.
%I would like to thank Bernhard Heim, Piotr Miska, Markus Neuhauser and Maciej Ulas for their time, profound comments and valuable suggestions. This research was funded by both a grant of the National Science Centre (NCN), Poland, no. UMO-2019/34/E/ST1/00094 and a grant from the Faculty of Mathematics and Computer Science under the Strategic Program Excellence Initiative at the Jagiellonian University in Kraków.

\end{document}